\documentclass[10pt, a4paper,reqno]{preprint}
\usepackage{amsfonts,amsmath,amsthm, times}

\newtheorem{theorem}{Theorem}

\newtheorem{proposition}[theorem]{Proposition}
\newtheorem*{proposition*}{Proposition}
\newtheorem*{corollary*}{Corollary}
\newtheorem{remark}[theorem]{Remark}
\newtheorem{corollary}[theorem]{Corollary}

\newtheorem*{observation*}{Observation}

\newtheorem*{example*}{Example}

\numberwithin{equation}{section}

\begin{document}

\newcommand{\tmname}[1]{\textsc{#1}}
\newcommand{\tmop}[1]{\operatorname{#1}}
\newcommand{\tmsamp}[1]{\textsf{#1}}
\newenvironment{enumerateroman}{\begin{enumerate}[i.]}{\end{enumerate}}
\newenvironment{enumerateromancap}{\begin{enumerate}[I.]}{\end{enumerate}}

\newcounter{problemnr}
\setcounter{problemnr}{0}
\newenvironment{problem}{\medskip
  \refstepcounter{problemnr}\small{\bf\noindent Problem~\arabic{problemnr}\ }}{\normalsize}
\newenvironment{enumeratealphacap}{\begin{enumerate}[A.]}{\end{enumerate}}
\newcommand{\tmmathbf}[1]{\boldsymbol{#1}}

\def\paral{/\kern-0.55ex/}
\def\parals_#1{/\kern-0.55ex/_{\!#1}}
\def\bparals_#1{\breve{/\kern-0.55ex/_{\!#1}}}
\def\n#1{|\kern-0.24em|\kern-0.24em|#1|\kern-0.24em|\kern-0.24em|}

\newcommand{\A}{{\bf \mathcal A}}
\newcommand{\B}{{\bf \mathcal B}}
\def\C{\mathbb C}
\newcommand{\D}{{\rm I \! D}}
\newcommand{\dom}{{\mathcal D}om}
\newcommand{\pathR}{{\mathcal{\rm I\!R}}}
\newcommand{\Nabla}{{\bf \nabla}}
\newcommand{\E}{{\mathbb E}}
\newcommand{\Epsilon}{{\mathcal E}}
\def\f{\frac}
\newcommand{\F}{{\mathcal F}}
\newcommand{\G}{{\mathcal G}}
\def\g{{\mathfrak g}}
\newcommand{\HH}{{\mathbb H}}
\def\h{{\mathfrak h}}
\def\k{{\mathfrak k}}
\newcommand{\I}{{\mathcal I}}
\def\LL{{\mathbb L}}
\def\law{\mathop{\mathrm{ Law}}}
\def\m{{\mathfrak m}}
\newcommand{\K}{{\mathcal K}}
\newcommand{\p}{{\mathbb P}}
\newcommand{\R}{{\mathbb R}}
\newcommand{\Rc}{{\mathcal R}}
\def\T{{\mathcal T}}
\def\M{{\mathcal M}}
\def\N{{\mathcal N}}
\newcommand{\pnabla}{{\nabla\!\!\!\!\!\!\nabla}}
\def\X{{\mathbb X}}
\def\Y{{\mathbb Y}}
\def\L{{\mathcal L}}
\def\1{{\mathbf 1}}
\def\half{{ \frac{1}{2} }}
\def\vol{{\mathop {\rm vol}}}

\def\ad{{\mathop {\rm ad}}}
\def\Conj{{\mathop {\rm Ad}}}
\def\Ad{{\mathop {\rm Ad}}}
\newcommand{\const}{\rm {const.}}
\newcommand{\eg}{\textit{e.g. }}
\newcommand{\as}{\textit{a.s. }}
\newcommand{\ie}{\textit{i.e. }}
\def\s.t.{\mathop {\rm s.t.}}
\def\esssup{\mathop{\rm ess\; sup}}
\def\Ric{\mathop{\rm Ric}}
\def\div{\mathop{\rm div}}
\def\grad{\mathop{\rm grad}}
\def\ker{\mathop{\rm ker}}
\def\Hess{\mathop{\rm Hess}}
\def\Image{\mathop{\rm Image}}
\def\Dom{\mathop{\rm Dom}}
\def\id{\mathop {\rm Id}}
\def\Image{\mathop{\rm Image}}
\def\Cyl{\mathop {\rm Cyl}}
\def\Conj{\mathop {\rm Conj}}
\def\Span{\mathop {\rm Span}}
\def\trace{\mathop{\rm trace}}
\def\ev{\mathop {\rm ev}}
\def\Ent{\mathop {\rm Ent}}
\def\tr{\mathop {\rm tr}}
\def\graph{\mathop {\rm graph}}
\def\loc{\mathop{\rm loc}}
\def\so{{\mathfrak {so}}}
\def\su{{\mathfrak {su}}}
\def\u{{\mathfrak {u}}}
\def\o{{\mathfrak {o}}}
\def\pp{{\mathfrak p}}
\def\gl{{\mathfrak gl}}
\def\hol{{\mathfrak hol}}
\def\z{{\mathfrak z}}
\def\t{{\mathfrak t}}
\def\<{\langle}
\def\>{\rangle}
\def\span{{\mathop{\rm span}}}
\def\diam{\mathrm {diam}}
\def\inj{\mathrm {inj}}
\def\Lip{\mathrm {Lip}}
\def\Iso{\mathrm {Iso}}
\def\Osc{\mathop{\rm Osc}}
\renewcommand{\thefootnote}{}
\def\V{\mathbb V}

\newcommand{\term}{{T}}
\newcommand{\tm}{{t}}
\newcommand{\stm}{{s}}
\newcommand{\pole}{{y_0}}

\author{Xue-Mei Li}

\title{Strict Local Martingales: Examples}
\institute{Mathematics Institute, The University of Warwick, Coventry, U.K.}
\date{}
\maketitle

\begin{abstract}{We show that a continuous local martingale is a strict local martingale if its supremum process is not in $L_\alpha$ for a positive number $\alpha$ smaller than $1$. Using this we construct a family of strict local martingales which are not Bessel processes.}
\end{abstract}\\[1em]


\\[1em]


\section{Introduction} Gauged by the small moments of the supremum of a local martingale, we
determine whether a stochastic integral is a strict local martingale. We construct a family of
strict local martingales, using neither Bessel processes nor the speed measures of one dimensional diffusions. This can be considered as an addendum to  \cite{E-L-Y-99, E-L-Y-97}, a work began around the period Marc Yor used to visit Coventry with his football team and  I just completed my thesis and  was obsessed with martingales. 
I have many fond memories of visiting Marc at Jussieu and   in St Cheron, discussing mathematics,  life, and the universe.

The  initial motivation for studying local martingales was to check the effectiveness of the criterion given in \cite{flow} for strong $p$-completeness of a stochastic differential equation (SDE) on a $d$-dimensional space (assuming suitable regularity, e.g. $C^2$ driving vector fields). Roughly speaking, an SDE is strongly $p$-complete if its solution flow moves a  $p$-dimensional sub-manifold into another one, without breaking it. Strong $p$-completeness is weaker than strong completeness by the latter we mean that the solution is continuous with respect to the initial data and time for all time. If the SDE and its adjoint SDE (the one with the drift given the negative sign) is also strongly complete then the solutions flow are diffeomorphisms for almost all $\omega$.  Strong (d-1) completeness is equivalent to strong completeness and strong 1-completeness allows differentiating the solution flow and their corresponding semi-group with respect to the initial data.  The criterion for the strong $p$-completeness, provided that the solution to the SDE from a particular initial value does not explode, 
 is \cite[Thm. 4.1]{flow}: \small$$\sup_{x \in K}\E\left( \sup_{s\le t} |D\phi_s|^{p+\delta}\chi_{t<\zeta(x)}\right)<\infty,$$ \normalsize
where $\phi_s(x)$ is the solution, $D\phi_s$ its derivative with respect to the initial data (solution to the linearised SDE), $\zeta(x)$ its life time, $K$ any compact subset and $\delta>0$. 

 For an explicit  SDE in \cite{flow},  which  is strongly $(d-2)$ complete but not strongly $(d-1)$ complete (the only example known to the author),
this finally comes down to for which values of $\alpha$, $\E\left( \sup_{s\le t} |x+B_s|^\alpha\right)$ is finite where $(B_s)$ is a $d$-dimensional Brownian motion.
It turns out that for a local martingale which is not martingale (nicknamed strict(ly) local martingales) there is a dichotomy for a function $F\in W^{1,1}$: $\E F(\sup_{s\le t} M_s)<\infty$ if and only if $\int_\epsilon^\infty \frac {F'(y)}{y}dy<\infty$ for some $\epsilon\ge 0$, see \cite[Prop. 2.3, also pp 332]{E-L-Y-99}, the article took a long while to complete (our second study was completed first), but it does  include a generalisation of this integrability criterion  to a class of semi-martingales, \cite[Prop. 3.5]{E-L-Y-99}, proved by a perturbative method.

 The popularity of strict local martingales is largely due to Freddy Delbaen and Walter Schachermayer's paper \cite{Delbaen-Schachermayer-95} on incomplete markets, and due to Alexander Cox and David Hobson \cite{Cox-Hobson}, Robert Jarrow, Philip Protter and K. Shimbo \cite{Jarrow-Protter-Simbo}  for the use of strict local martingales in bubble modelling. See also M. Loewenstain and G. Willard  \cite{Loewenstein-Willard} and Philip Protter's excellent survey \cite{Philip-13}.    The incomplete market problem concerns equivalent martingale measures and free lunch may follow from
 a strict local martingale measure.  A classification in terms of the bubble time of a bubble asset  is given by Robert Jarrow, Philip Protter and K. Shimbo \cite{Jarrow-Protter-Simbo}. A type three financial bubble models a non-zero asset price with bounded life time 
 bursting at on or before the bubble time and is a strict local martingale. The two are linked as following: birth of bubbles are impossible in complete market,
and possible in incomplete markets, see   Philip Protter  \cite[section3]{Protter}.

 Given a function $m(t)$ of finite variation there exists a strict local martingale $(M_t)$ with $\E M_t=m(t)$ which is a time changed 3-dimensional Bessel processes with time change $r^{-1}(m(t))$ where $r(t)$ is the mean process of the Bessel process,  see \cite[Corollary 3.9]{E-L-Y-99} where constructions are also given using the speed measure. 
 However if the negative part or the positive part of $(M_t)$ behaves well, then $\E(M_t)$ cannot vanish for all $t$, see \cite[Lemma 2.1]{E-L-Y-99}, the case of a positive local martingale is classic. A positive strict local martingale has even nicer properties, Soumik Pal and Philip Protter \cite{Pal-Philip}, following \cite{Delbaen-Schachermayer-95}, showed that they can be obtained as the reciprocal of a martingale under an $h$-transform.
The fact that $(M_t)$ where $M_0=0$  is strictly local if, and only if, $\E M_t$ vanishes has implications
in probabilistic representations of solutions of PDE's and useful for the identification of the domain of the generator of a Markov process.

 Beautiful and complete answers are given to the question whether the solution to a  one dimensional SDE without drift is a local martingale, which we discuss below.
Let $(M_t)$ be a solution to the equation $M_t=x+\int_0^t a(M_s)dW_s$ (where $a^2>0$, $a^{-2}\in L_{loc}^1$ and $(W_t)$ a one dimensional Brownian motion), then 
 it is a strict local martingale if and only if $\int_\epsilon^\infty \f {x}{a^2(x)} dx=\infty$ for some $\epsilon>0$.
 This clean solution is given in the articles by  Freddy Delbaen and H. Sgirakawa \cite[Thm1.6]{Delbaen-Shirakawa}, and Shinichi Kotani \cite[Thm 1]{Kotani},  see also Hardy Hulley and Eckhard Platen \cite[Thm1.2]{Hulley-Platen} for a description by the first passage times of $(M_t)$.
Also, in \cite{Mijatovic-Urusov}, Alexsander Mijatovi\'c and Mikhail Urusov  solved the problem whether the exponential
martingale of a stochastic integral is a strict local martingale. These approaches explore the fact that the non-explosion problem for one  dimensional elliptic SDE is determined by the Feller test and  is equivalent to that the exponential local martingale in the Girsanov transform removing the drift is a martingale. 

Despite of the success in classifying one dimensional diffusions which are also local martingales, there is wanting in concrete examples of  strict local martingales, especially in  variety. The strict local martingales we will construct later are based on an entirely different approach. 

{\it Acknowledgement. } I met Freddy in Anterwerp in February 1994, who was in conversation with Marc and came to talk to me in  `The Physics and Stochastic Analysis Conference' organised by Jan van Casteren.  I  would like to take this opportunity to thank all of them. The example itself was constructed quite many years ago. I had the opportunity to tell Marc about it, who was enthusiastic and thought it should be published. It did take me many years to finally write it up. This paper  is dedicated to the friendship and fond memory of Marc, who always loved the beauty of simplicity.
\section{Examples}   
 The observation below,  presumably known although I know no reference for it, is that  the sample paths of a strictly local martingale oscillate faster than that of a martingale, 
the wild oscillation might explain why strict local martingales are useful for bubble modelling. Recall that a stochastic process $(X_t)$ is of class $DL$ if
 $\{X_S\}$ where $S$ ranges through bounded stopping times is uniformly integrable and that  a local martingale $(M_s)$ satisfying that $\sup_{s\le t} |M_s|$ is integrable is a martingale. 
 Furthermore, a local martingale with $M_0$ integrable and with its negative part  of class $DL$ is a super-martingale, and it is  a martingale if and only if its mean value is constant in time \cite[Prop. 2.2]{E-L-Y-99}.
 Throughout the paper the underlying probability space 
$(\Omega, \F, \F_t, P)$ satisfies the standard assumptions: right continuity and completeness.
\begin{proposition}\label{prop1}
Let $(M_t)$ be a right continuous local martingale, $T$ a finite stopping time  and $M^T$ the stopped process.
\begin{enumerate}
\item Suppose that  $\E \left(\sup_{s\le t} |M_s|^\alpha\right)=\infty$ for some number $\alpha\in (0,1)$.  Then 
$(M_s, s\le t)$ is a  strict local martingale.
\item Suppose that  $\E \left(\sup_{s\le T} |M_s|^\alpha\right)=\infty$ for some number $\alpha\in (0,1)$ and $T\le t_0$. Then $(M^T_s, s\le t_0)$ is a strict local martingale (also
 $\limsup_{\lambda\to \infty} \lambda \p\left(\sup_{s\le T}|M_s|\ge \lambda\right)=\infty$). If $\lim_{t\to \infty}M_t$ exists we may take $T=\infty$.
\end{enumerate}  
\end{proposition}

\begin{proof}
Let $T$ be a stopping time. Suppose that $\E \left(\sup_{s\le T} |M_s|^\alpha\right)=\infty$ for some $\alpha \in (0,1)$, then
\begin{equation}
\label{1} \limsup_{\lambda\to \infty} \lambda \p\left(\sup_{s\le T}|M_s|\ge \lambda\right)=\infty.
\tag{*}
\end{equation}
 For otherwise, there exists $K$ such that 
 $ \p\left(\sup_{s\le T }|M_s|\ge \lambda\right)\le \f K \lambda$ and 
 $$\E  \left(\sup_{s\le T} |M_s|^\alpha\right)
 =\int_0^\infty \p\left(\sup_{s\le T}|M_s|^\alpha\ge \lambda\right) d\lambda
  \le 1+\int_1^\infty \f K {\lambda^{\f 1 \alpha}} d\lambda<\infty.
 $$ 
 If $(M_s, s\le t)$ were a martingale, we apply the maximal inequality to $M$ and obtain, for $\lambda>0$, \small
$$\E |M_t|\ge \lambda \p\left(\sup_{s\le t} |M_s|\ge \lambda\right),$$
\normalsize
contradicting with (*), completing the proof for part (1).
If $(M^T_t, t\le t_0)$ were a martingale, we apply the maximal inequality to $M^T$ and obtain, for $\lambda>0$, \small
$$\E |M^T_{t_0}|\ge \lambda \p\left(\sup_{s\le t_0}|M_s^T|\ge \lambda\right)=\lambda \p\left(\sup_{s\le T}|M_s|\ge \lambda\right),$$
\normalsize
this contradicts with (*), concluding part (2).
   \end{proof}
   From the proof we see also the following statement.
  \begin{remark}
  \label{remark1}
A right continuous local martingale $(M_s, s\le t)$ satisfying the condition  $\limsup_{\lambda\to \infty} \lambda \p\left(\sup_{s\le t}|M_s|\ge \lambda\right)=\infty$ is a strict local martingale.
  \end{remark}
In Proposition \ref{prop1} we stated a criterion for the strict localness of a local martingale in terms of small moments of the supremum process. 
Before using the proposition to construct an explicit strict local martingale, we observe  that  there is a class of local martingales for which such small moments are always finite. In this class, which we describe more carefully in the remark below, is the strict local martingales 
$(R_t)^{2-\delta }$ where $\delta>2$ and $R_t$ is a-$\delta$ dimensional Bessel process. See Corollary 2.4 in \cite{E-L-Y-99}.

\begin{remark}\label{remark}
Suppose $(M_s, s\le t)$ is a continuous positive local martingale with $M_0\in L_1$.
  Then $(M_s, s\le t)$ is $L^1$ bounded, $\E \left(\sup_{s\le t} |M_s|^\alpha\right)<\infty$ for any $\alpha \in (0,1)$,
  and  $\lim_{\lambda\to \infty} \lambda \p\left(\sup_{s\le  t}|M_s|\ge \lambda\right)=\E(M_0-M_t)$. 
So $(M_s, s\le t)$ is strictly local  if and only if the latter limit does not vanish for some $t$. 
\end{remark}
\begin{proof}
To see this,  we apply  Lemma 2.1 in \cite{E-L-Y-99} to conclude that
$(M_s, s\le t)$ is $L^1$ bounded and
$\lim_{\lambda \to \infty} \lambda \p\left(\sup_{s\le t}M_s\ge \lambda\right)
=\E M_0-\E M_t<\infty$.
Since $M_t\ge 0$,  the argument in the proof for part (1) shows that $\E \left(\sup_{s\le t} |M_s|^\alpha\right)$ is finite for any $\alpha\in (0,1)$. 
\end{proof}

  The construction below for strict local martingales depends crucially on the Burkholder-Davies-Gundy inequality for small values $p<1$.  
 For a c\'adl\'ag local martingale, BDG inequalities are only known to hold with values $p$ greater or equal to one. The sample continuity assumption is 
 hence essential for the statement below, on the other hand a c\'adl\'ag local martingale with  a.s. finite total jumps is the sum of a continuous local martingale and a finite variation process, and one can construct strict local martingales with jump by adding such a jump process.   Our method does not seem to apply to stochastic integrals driven by a L\'evy process, for which please see the papers by Philip Protter \cite{Protter} and that by Constantinos Kardaras, D\"orte Kreher, and Ashkan Nikeghbali \cite{Kardaras-Doerte-Nikeghbali}.


\begin{corollary}\label{corollary}
Let $f$ be a progressively measurable function with values in $\R^d$. 
 If $\E\left( \int_0^t |f(s)|^2ds\right)^{\alpha/2}=\infty $ for some number $\alpha\in (0,1)$ then $\int_0^t\<f(s),dB_s\>$ is a strict local martingale.
\end{corollary}
\begin{proof}
Set $M_t=\int_0^t\<f(s),dB_s\>$. By Burkholder-Davis-Gundy inequality,
\small $$\E \left(\sup_{s\le t} |M_s|^\alpha\right) \ge C_\alpha
 \E\left( (\<M,M\>_t)^{\alpha/2}\right),$$ \normalsize
 which is infinite by the assumption,  apply part (1) of Proposition \ref{prop1} to conclude.
\end{proof}

Let  $(W_t)$ be a one dimensional Brownian motion.  Then stochastic integrals of the form $\int_0^t g(s)dW_s$ is a strict
local martingale if $\E\left( \int_0^t|g(s)|^\alpha ds\right)=\infty$ for some number $\alpha<1$. This follows
 from H\"older's inequality: $\|g\|_{ L_\alpha( \Omega, L_2([0,t])}\ge \|g\|_{L_\alpha(\Omega; L_\alpha ([0,t]))}$.
It is immediate that on $[0, t_0]$ where $t_0<1/4$,
 $\int_0^t e^{|W_s|^2} dW_s$ is an $L^2$ martingale (Fernique's theorem); while $\int_0^t e^{|W_s|^3} dW_s$ is a strict local martingale on any interval.
\\

Finally we remark on the characterisation of a continuous  local martingale being a martingale. Assume $M_0\in L_1$.
We have seen that if  $$\ell(t):=\limsup_{\lambda\to \infty} \lambda \p\left(\sup_{s\le t}|M_s|\ge \lambda\right)$$ is infinite then $(M_t)$ is a strict local martingale. On the other hand if $\ell(t)=0$ and if $(M_t)$ is $L^1$ bounded, then $(M_t)$ is a uniformly integrable martingale.
  See K. Murali Rao \cite{Rao},  J. Az{\'e}ma,  R. F. Gundy and M. Yor \cite{Azema-Gundy-Yor}. From the point of view of mathematical finance the class of
 strict local martingales with $\ell(t)$ finite has been often studied.
 Also, suppose that the negative part of $(M_t)$ is of class $D$ and $M_t\in L^1$, then 
$\gamma(t) :=\limsup_{\lambda\to \infty} \lambda \p\left(\sup_{s\le t} M_s\ge \lambda\right)$ is always a finite number that could vanish in which case we have a martingale or that is a non-zero  number in which case we have a strict local martingale. Moreover $\gamma(t)=\E[M_0-M_t]$, see \cite{E-L-Y-97}, and \cite{E-L-Y-99}. If $(M_s, s\le t)$, where $t$ is a finite number, is a martingale then it is uniformly integrable and so $\ell(t)=\gamma(t)=0$. This may be confusing, after all we have been warned a martingale may not be $L^1$ bounded, $L^1$ bounded martingales may not be uniformly integrable! Such caution needs only be taken if the martingale is defined on an infinite time horizon or on an open interval $[0, t)$. (However we should pay attention to the problem whether the supremum of the martingale is integrable.) 
 Recently, Hardy Hulley and Johannas Ruf \cite{Hulley-Ruf} studied a necessary and sufficient condition for a suitable class of local martingales with jumps to be a martingale. 
We conclude this paper with the following open question.

{\bf Open Question.} What is a suitbale analogous result for local martingales with jumps?

\def\polhk#1{\setbox0=\hbox{#1}{\ooalign{\hidewidth
  \lower1.5ex\hbox{`}\hidewidth\crcr\unhbox0}}}

\end{document}